\documentclass{article}
\usepackage{graphicx} 
\usepackage{amsthm, amsmath, amssymb, amsfonts, url, booktabs, tikz, setspace, fancyhdr, bm}
\usepackage{hyperref, comment}
\usepackage{geometry}
\usepackage{hyperref, enumerate}
\usepackage[shortlabels]{enumitem}
\usepackage[babel]{microtype}
\usepackage[english]{babel}
\usepackage[capitalise]{cleveref}

\newtheorem{defi}{Definition}

\newtheorem{cor}[defi]{Corollary}

\newtheorem{prop}[defi]{Proposition}
\newtheorem{prob}[defi]{Problem}
\newtheorem{exam}[defi]{Example}
\newtheorem{q}[defi]{Question}

\newtheorem{claim}[defi]{Claim}
\newcommand*{\myproofname}{Proof}
\newenvironment{claimproof}[1][\myproofname]{\begin{proof}[#1]}{\end{proof}}

\newcommand{\diff}{diff}

\newcommand*{\ceilfrac}[2]{\mathopen{}\left\lceil\frac{#1}{#2}\right\rceil\mathclose{}}
\newcommand*{\floorfrac}[2]{\mathopen{}\left\lfloor\frac{#1}{#2}\right\rfloor\mathclose{}}

\newcommand*{\abs}[1]{\left \lvert #1\right\rvert}
\newcommand{\diam}{diam}
\newcommand{\rad}{rad}

\title{Towards the essence of \v Solt\'es' problem}
\author{Stijn Cambie \thanks{Department of Computer Science, KU Leuven Campus Kulak-Kortrijk, 8500 Kortrijk, Belgium. Supported by a postdoctoral fellowship by the Research Foundation Flanders (FWO) with grant number 1225224N.}}
\date{}

\begin{document}

\maketitle

\begin{abstract}
We explore the question asking for graphs $G$ for which the total distance decreases, possibly by a fixed constant $k$, upon the removal of any of its vertices. We obtain results leading to intuition and doubts for the \v Solt\'es' problem ($k=0$) and its conjectures.
\end{abstract}

\section{Introduction}\label{sec:intro}

In 1991, \v Solt\'es~\cite{Soltes91} observed that if one removes a vertex of a cycle $C_{11}$, the total distance does not change, 
and asked whether there are other such graphs~$G$ (nowadays called \v Solt\'es' graphs).
It is one of the most elementary questions one can pose~\cite{KST23} related to $W(G)$, the total distance of $G$.
Infinitely many examples were given by Spiro~\cite{Spiro22} and Cambie~\cite{Cambie24} for the relaxation of \v Solt\'es' problem to signed graphs and  hypergraphs respectively.

Meanwhile, there are numerous open conjectures (weaker to stronger forms), conjecturing that \v Solt\'es' graphs are e.g. regular or vertex-transitive, and that $C_{11}$ is the only \v Solt\'es' graph. See~\cite[conj.~47, 48, 51]{KST23}.
One of the arguments, as can be seen in the conclusion of~\cite{BKS23}, was that no other examples were found among the $>10^8$ vertex-transitive graphs in the census by Holt and Royle~\cite{HR20}. We will give intuition why all of these (different from the cycles) satisfy $W(G)>W(G \setminus v)$ for $v \in V(G)$. 

In~\cref{sec:res_graphs}, we prove that if $W(G) \le W(G \setminus v)$ for every $v \in V(G),$ the diameter cannot be too small, and if the value $W(G) - W(G \setminus v)$ is independent of the vertex $v \in v(G),$ then $G$ is a cycle or the minimum degree of $G$ is at least $3.$

In~\cref{sec:conditionalconstruction} (and~\cref{sec:app}), we present eight graphs $G$ for which two third of the vertices $v$ satisfy $W(G \setminus v)=W(G)$ (these are \v Solt\'es' vertices), and give conditional examples of graphs for which an even higher fraction of the vertices satisfy the inequality.

In~\cite[Prob.~3]{AOVVVY23} (and~\cite[Prob.~44]{KST23}), the authors mentioned that arguments to solve \v Solt\'es' problem may also work for the following generalisation.
\begin{prob}\label{prob:extensionofSoltes}
    For a fixed \( z \in \mathbb{Z} \), find all graphs \( G \) for which the equality \( W(G) - W(G - v) = z \) holds for all vertices \( v \).
\end{prob}

As a first thing, we want to convince the reader that this problem is nearly impossible to solve in general\footnote{For fixed $z$, it might be possible to reduce to a finite number of candidates.} and by the reasoning of the authors, solving the \v Solt\'es' problem may be hard or impossible as well.


Of course, for every $z>0$, the graph $K_{z+1}$ is a solution, but there are values of $z$ for which there are many solutions. 
By an immediate argument using the pigeon hole principe on the number of cubic vertex-transitive graphs (which is of the form $n^{\Theta( \log n)}$ as estimated in~\cite{PSV17}), the number of solutions can be arbitrary large. 
For a concrete indication, the census by Holt and Royle~\cite{HR20} in \url{https://zenodo.org/records/4010122} contains more than $10^8$ vertex-transitive graphs with diameter bounded by $3$, each corresponding to a value $0<z<141$.

By focusing on a subfamily (defined later), for which the plausible values of $z$ can be both negative or positive, we can expect that also $z=0$ might have multiple solutions. We conjecture this is the case if there are e.g. infinitely many vertex-transitive graphs (different from cycles) for which $W(G)<W(G \setminus v)$ for all $v \in V(G).$

We add more doubts and remarks on the circulating conjectures in~\cref{sec:counterarg}, giving evidence why these conjectures may be false. 

In~\cref{sec:conc}, we give our conclusions and state our core question, which would lead to the essence behind \v Solt\'es' problem.
Calling a graph $G$ satisfying~\cref{prob:extensionofSoltes} for a $z \le 0$ a negative-\v Solt\'es' graph (these are far from vertex-robust for distances), we can summarize that the essential question is whether there are infinitely many negative-\v Solt\'es' graphs which are not cycles. 

\subsection{Notation and Terminology}\label{subsec:not&def}

We only consider simple connected graphs $G=(V,E).$
The degree of a vertex $v$, denoted by $\deg(v)$, is the number of edges containing $v$. The minimum and maximum degree, $\delta$ and $\Delta$, represent the smallest and largest degrees. When removing the vertex $v$ and its incident edges from $G$, we obtain $G \setminus v$.

The distance $d_G(u,v)$ or $d(u,v)$ between two vertices $u$ and $v$ is equal to the length of the shortest path between them. The diameter, $\diam(G)$, is the largest among all possible distances and the total distance (or Wiener index), $W(G)$, the sum over all of them; $W(G)=\sum_{u,v \in V} d(u,v)$. The radius $\rad(G)= \min_{u \in V} \max_{v \in V} d(u,v)$ gives the maximum distance from a central vertex.
The transmission of a vertex $v$, $\sigma(v)=\sum_{u \in V} d(u,v)$, is the sum of distances between $v$ and the other vertices.
We also define the arc-graph of $G$.
\begin{defi}
    Let $G=(V(G),E(G))$ be a graph.
    Let $A(G)$ be the arcs of $G$, i.e., the ordered pairs of neighbouring vertices.
    Note that $\abs{A(G)}=2\abs{E(G)}.$
    The \textbf{arc-graph} $H$ of $G$ has vertex set $A(G)$, with two arcs being adjacent if \begin{itemize}
        \item they correspond to the same edge, i.e., $(u,v)$ is adjacent to $(v,u)$ when $uv \in E(G),$
        \item the arcs have the same root, i.e., $(v,u)$ and $(v,w)$ are adjacent in $H$.
    \end{itemize}
    Equivalently, $H$ is equal to the line-graph of the subdivision of $G.$
\end{defi}

Essentially, ranging over all vertices, a vertex $v$ of degree $d$ is replaced by a clique $K_d$ whose endvertices are connected with one initial neighbour of $v$ each.
It has been introduced before as the subdivided-line graph in~\cite{H15}, and for cubic graphs, as the truncation of the graph~\cite{BKS23}.
An example, $A(K_4)$, and a sketch of the local replacement (for $d=7$) is shown in~\cref{fig:arcgraph_K4}.

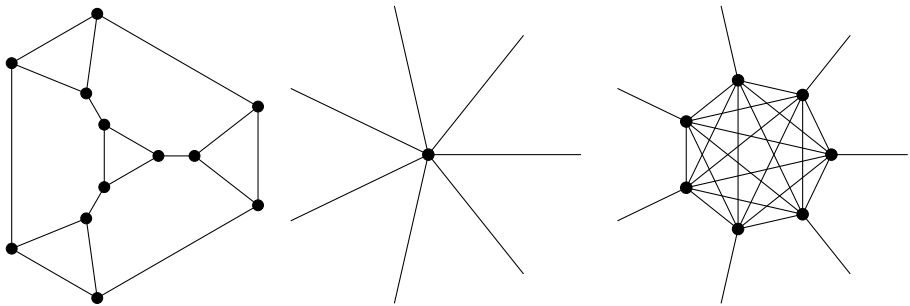
\begin{figure}[h]
    \centering
    \begin{tikzpicture}[scale=0.6*0.8]
\foreach \x in {0,2,4}{
\draw (60*\x+20:4)--(\x*60-20:4)--(\x*60:2)--cycle;
}

\draw (120:1)--(240:1)--(0:1)--cycle;

\foreach \x in {1,3,5}{
\draw (60*\x+40:4)--(\x*60-40:4);
}
\foreach \x in {0,1,2}{
\draw[fill] (\x*120:1) circle (0.15);
\draw[fill] (\x*120:2) circle (0.15);
\draw[fill] (\x*120+20:4) circle (0.15);
\draw[fill] (\x*120-20:4) circle (0.15);
\draw (\x*120:1)--(\x*120:2);
}
\end{tikzpicture}\quad
    \begin{tikzpicture}[scale=0.7*0.725]

\foreach \x in {0,1,2,...,6}{
\draw (0,0) -- (360/7*\x:4) ;
}

\draw[fill] (0,0) circle (0.15);

\end{tikzpicture}
 \quad
    \begin{tikzpicture}[scale=0.7*0.725]

\foreach \x in {0,1,2,...,6}{
\foreach \y in {1,2,3}{
\draw  (360/7*\x:2)--(360/7*\x+360/7*\y:2) ;
}
\draw  (360/7*\x:2)--(360/7*\x:4) ;
\draw[fill] (360/7*\x:2) circle (0.15);
}

\end{tikzpicture}
    
    \caption{Arc-graph of $K_4$ (left) and local modification of a degree $7$ vertex (middle and right)}
    \label{fig:arcgraph_K4}
\end{figure}

\section{Necessary condition for \v Solt\'es' graphs}\label{sec:res_graphs}

Forbidding an isolated vertex as a \v Solt\'es' graph, and remarking that if $u$ is a pendent vertex of a connected graph $G$, we have $W(G \setminus u)<W(G)$, we know that the minimum degree ($\delta$) of a \v Solt\'es' graph is at least $2$.
Here we prove that if $G \not= C_{11}$ is a \v Solt\'es' graph, then $\delta(G)\ge 3.$
This is a stronger version of a result by Dragan Stevanovi\'c (private communication), that a \v Soltes' graph different from $C_{11}$ cannot have internal paths with $3$ consecutive degree-$2$-vertices.
In the proofs, we denote $G_v$ for the graph $G \setminus v$.

\begin{prop}\label{prop_delta_ge3}
 A graph for which $\delta(G)=2$ and $W(G \setminus v)$ is constant for all $v \in V(G),$ is a cycle.
\end{prop}

\begin{proof}
 Let $G=(V,E)$ be such a graph with $\delta(G)=2$, let $p$ be a vertex with degree $2$ and $a,b$ be its two neighbours.
 We need $W(G)=W(G_a)=W(G_p)=W(G_b).$
 Let $X=V \setminus \{a,p,b\}.$
 Let $H_a=G[X \cup a] $ and $H_b=G[X \cup b].$
 Let $\sigma(a)= \sum_{w \in X} d_{H_a}(a,w)$ and $\sigma(b)= \sum_{w \in X} d_{H_b}(b,w).$
 
 We are now ready to compare $W(G_a), W(G_p)$ and $W(G_b).$
 For this, we split these total distances in multiple parts.
 First observe that
 $d_{G_p}(a,b)\le \lvert X \rvert +1$ with equality if and only if $G=C_n,$ while
 $d_{G_a}(b,p)=d_{G_b}(a,p)=1.$
 Next, we note that for every choice of $x, y \in X$,
 $$d_{G_p}(x,y) \le d_{G_a}(x,y), d_{G_b}(x,y).$$
 
 Finally, since $H_a, H_b \subset G_p$ we note that
 $$\sum_{w \in X} \left( d_{G_p}(a,w) + d_{G_p}(b,w) \right) \le \sigma(a)+\sigma(b) $$
 $$\sum_{w \in X} \left( d_{G_a}(p,w) + d_{G_a}(b,w) \right) = 2\sigma(b)+\lvert X \rvert $$
 $$\sum_{w \in X} \left( d_{G_b}(a,w) + d_{G_b}(p,w) \right) = 2\sigma(a)+\lvert X \rvert. $$
 
 Using these observations, we conclude that 
 \begin{align*}
 0=&2W(G_p)-W(G_a)-W(G_b)\\
 =&2d_{G_p}(a,b)-d_{G_a}(p,b)-d_{G_b}(a,p)
 +\sum_{x,y \in X} \left( 2d_{G_p}(x,y) - d_{G_a}(x,y)- d_{G_b}(x,y) \right)\\
 &+\sum_{w \in X} \left(
 2d_{G_p}(a,w) + 2d_{G_p}(b,w)
 -d_{G_a}(p,w) - d_{G_a}(b,w)
 -d_{G_b}(a,w) -d_{G_b}(p,w) \right)\\
 \le& 2d_{G_p}(a,b)-2
 -2 \lvert X \rvert\\
 \le& 0.
 \end{align*}
 Since equality need to be attained in every step, we have 
 $d_{G_p}(a,b)=1+ \lvert X \rvert,$ which implies that $G_p$ is a path from $a$ to $b$, and thus $G$ is a cycle.
\end{proof}

\begin{prop}\label{prop:diam_ge3}
 A graph $G$ of order $n>1$ with $\diam(G)\le 2$ has at least one vertex $v$ for which either $G \setminus v$ is disconnected, or $W(G \setminus v)<W(G).$
\end{prop}

\begin{proof}
This is obvious for a clique (diameter $1$ graph).
 So now assume $\diam(G)=2.$
 First we observe that $G$ cannot have a universal vertex $v$, since for a vertex $u \not= v$ $W(G)=2\binom{n}{2}-m>2\binom{n-1}{2}-m+\deg(u)=W(G_u).$
 So $\rad(G)>1$, i.e. for every $v \in V$ its closed neighbourhood $N[v]$ is not equal to $V.$
 The only pairs of vertices $u,u' \in V(G_v)$ for which $ d_{G_v}(u,u')> d_{G}(u,u'),$ have to be both neighbours of $v$ in $G$ (and have no other common neighbour).
 
 Note that for every $u\in N(v)$ and $w \not \in N[v]$, 
 $ d_{G_v}(u,w)= d_{G}(u,w)\le 2. $
 This implies that $ d_{G_v}(u,u')\le d_{G_v}(u,w)+d_{G_v}(w,u')\le 4$
 and thus $ d_{G_v}(u,u')- d_{G}(u,u')\le 2.$
 
 Since $\sum_{u \in V \setminus v} d(u,v)=2(n-1)-\deg(v) \ge n$, has to be lower bounded by $\sum_{ u,u'\in V \setminus v} \left( d_{G_v}(u,u')- d_{G}(u,u') \right),$
 there are at least $\frac{n}{2}$ pairs of vertices $(u,u')$ for which the unique shortest path uses $v$.
 
 Summing over all $v \in V$, this implies that there are at least $\frac{n^2}{2}>\binom{n}{2}$ pairs of vertices at distance $2$ in $G$, which is a contradiction.
\end{proof}

For self-centric graphs (graphs with $\diam=\rad$), e.g. vertex-transitive graphs, one can also exclude the diameter to be $3.$

\begin{prop}\label{prop:diam_ge4}
 A graph $G$ of order $n>1$ with $\diam(G)=\rad(G)= 3$ has at least one vertex $v$ for which either $G \setminus v$ is disconnected, or $W(G \setminus v)<W(G).$
\end{prop}

\begin{proof}
    Assume $G$ has connectivity at least two, i.e., $G \setminus v$ is connected for every $v \in V(G)$. 

    Let $v \in V(G)$ and let $v'$ be an antipodal vertex, that is, $d(v,v')=3.$
    If $a,b$ are two vertices with $d(a,b)=d(a,v)+d(v,b)=2$, we note that $d_{G_v}(a,b)\le d_{G_v}(a,v')+d_{G_v}(b,v')\le 2 \cdot 3=6.$

    If $a,b$ are two vertices with $d(a,b)=3$, $d(a,v)=1$ and $d(v,b)=2$,
    we can consider a shortest path between $a$ and $v'$ and let $w$ be the neighbour herein of $v'.$
    Then $d_{G_v}(a,b)\le d_{G_v}(a,w)+d_{G_v}(b,w)\le 2 +3=5.$
    
    Assume there are $x$ pairs of vertices at distance $2$, and $y$ pairs of vertices at distance $3$.
    If $d(a,b)=2,$ there is at most one vertex $v$ for which $d_{G_v}(a,b)>d(a,b)$ and the increase is at most $6-2=4$.
    If $d(a,b)=3,$ there are at most two vertices whose removal increase the distance between $a$ and $b$, hereby the increase is bounded by $5-3=2.$
    The total increase, $\sum_{ u,u'\in V \setminus v} \left( d_{G_v}(u,u')- d_{G}(u,u') \right),$ is thus bounded by $4(x+y).$
    On the other hand, the sum of the transmissions over all vertices, which equals $2W(G)$, is strictly above $2(2x+3y).$
    As such, there is a vertex $v$ for which $W(G \setminus v)<W(G).$
\end{proof}

\begin{prop}\label{prop:twinfree}
 A graph $G$ for which $W(G)\le W(G_v)$ for every $v \in V(G)$ is twin-free. More generally, there are no $v,v'\in G$ for which $N(v) \subseteq N(v').$
\end{prop}

\begin{proof}
 If not, for every $2$ vertices $x,y \in G_v= G \setminus v$ for which there exists a shortest path in $G$ using $v$, there is also one which uses $v'.$
 Thus $d_G(x,y)=d_{G_v}(x,y)$ for every $x,y \in G_v$, implying that $W(G)>W(G_v)$. 
\end{proof}

\section{A conditional construction for partial \v Solt\'es' graphs }\label{sec:conditionalconstruction}

An $\alpha$-\v Solt\'es' graph, is a graph for which at least an $\alpha$-fraction of its vertices are \v Solt\'es' vertices.
In~\cite[abstract]{BKS23} and~\cite[Sec.~6]{KST23}, it is mentioned that the only known $\frac 13$-\v Solt\'es' graphs were truncated cubic graphs at that point (neglecting the quartic example from~\cite{BKS23} obtained by taking the line graph of a certain cubic graph).
Meanwhile, by~\cite{AOVVVY23,DV24} more examples are known.
In this section, we give some other (conditional) examples.
We start by the following explicit example, which will be informative for a conditional construction with larger $\alpha$.

\begin{exam}
    There exists a modified quartic ($4$-regular) $\alpha$-\v Solt\'es' graph for $\alpha =\frac{27}{56} = \frac 12 -\frac 1{56}.$
\end{exam}

\begin{proof}
    For this, we take the arc-graph $G$ of the Moore graph $M$ with degree/ valency $4$ and girth $12$/ diameter $6$. 
    Note that the order of $M$ is $n(M)=2 \sum_{i=0}^{5}3^i=3^6-1=728$ and $n(G)=4(3^6-1)=2912.$
    For a fixed edge $e \in G,$ we have that there are $2\cdot 3^5\cdot3=2\cdot 3^6=1458$ many vertices at distance $11$ from $e.$
    For every $v \in V(G),$ $W(G)-W(G\setminus v)=29896-37356=-7460.$
    Noting that $7460=23+ \sum_{k=12}^{122} k, $
    we finally construct a graph $Q$ by connecting an end vertex $v_{12}$ of a $P_{111}$ (with vertices $v_k, 12 \le k \le 122$) to both endvertices of an edge $e$ of $G$, and finally add a pendent vertex $u_{23}$ to $v_{22}.$
    Then $n(Q)=2912+112=3024$
    and $W(Q)-W(Q \setminus v)=0$ for every $v \in V(G) \subset V(Q)$ which is at distance $11$ from $e.$
    This gives a $\frac{1458}{3024}=\frac{27}{56}$ ratio of vertices of $Q$ which are \v Soltes' vertices.
\end{proof}


The following proposition gives some intuition that one can expect that a large graph exists of which more than a $\frac 23$-fraction of its vertices are \v Solt\'es' vertices, which would answer~\cite[Prob.~1]{DV24}. 
Nevertheless, it is related with the very hard open problem on estimating the minimum order of a cage.
Note that in~\cite{EJMS19} a lower bound on the difference has been shown that is negligible towards our aims, but no good upper bound is known either~\cite{LUW96}.

The underlying idea is broader; if there is a large proportion of the vertices for which $W(G\setminus v)-W(G)$ equals the same positive constant, and we can add a little substructure to compensate evenly over them, then there is a graph with a large proportion of \v Solt\'es' vertices.

\begin{prop}
    If there exists a $7$-regular graph with even girth $g \ge 58$ which
    is arc-transitive and has roughly the Moore bound $N(7,g)$ many vertices, e.g. order no more than $1.0005 N(7,g)$, 
    then there are graphs with roughly a $\frac 57$\footnote{The exact value depends on how well the Moore bound is approximated} fraction \v Solt\'es' vertices.
\end{prop}

\begin{proof}
    Assume such a graph $G$ exists, for girth $g=2k$.
    Remember that $N(7,g)=2\sum_{i=0}^{k-1} 6^i = 2\frac{6^{k}-1}{6-1}=2\frac{6^{k}-1}{5}.$
    Let $uv$ be an edge of $G$. Note that the edges at distance at most $k-1$ from either $uv$ in the line graph of $G$ span a tree $T$ (in $G$, with internal vertices having degree $7$).
    
    Let $G'=A(G)$ be the arc-graph of $G.$
    This graph is vertex-transitive, since $G$ was arc-transitive.

    Due to the girth condition of $G$, every cycle in $G'$ that is not part of a $K_7$, has length at least $2g.$
    For every $2$ vertices $w_1,w_2$ in $G'$ for which the shortest path (which has length $i$) between them uses a vertex $v' \in V(G')$, the shortest path between them in $G \setminus v'$ has length at least $2g-i.$
    If $i \le g-1,$ we know that this was initially the unique shortest path. Thus the difference of distances changed by at least $(2g-i)-i=2(g-i)$.
    Let $u'$ be the unique neighbour of $v'$ which is not part of the same $K_7$.
    Now we can compute a lower bound for 
    $$\diff(G',v')=\sum_{w_1, w_2 \in G'\setminus v'} d_{G'\setminus v'}  (w_1,w_2)- d_{G'}(w_1,w_2)$$
    by counting the number of pairs of vertices whose shortest paths use $v'$ and are at distance less than $g$ initially.
    Let $N_{u'}$ be the set of vertices which are closer to $u'$ than to $v'$, and similarly $N_{v'} =\{ w \in V(G') \colon d_{G'}(w,v')< d_{G'}(w,u') \}.$
    Let $N_{v'}^i=\{ w \in N_{v'} \colon d(v', w)=i\}$ and define $N_{u'}^i$ analogously.
    Note that $\abs{N_{v'}^{j}}=6^{ \ceilfrac j2}$ for every $0\le j \le g.$
    The number of pairs of vertices at even distance, $2i+2<g$, whose shortest path between them uses $v'$, is now equal to 
    $$\sum_{j=0}^{2i} \abs{ N_{u'}^j} \abs{ N_{v'}^{2i+1-j}} = (2i+1)\cdot 6^{i+1}. $$
    Similarly, for odd distance $2i+1$, we get $$\sum_{j=0}^{2i-1} \abs{ N_{u'}^j} \abs{ N_{v'}^{2i-j}} = i\cdot (6^{i}+6^{i+1})=7i\cdot 6^{i}. $$

    This results in

    \begin{align*}
        \diff(G',v')&\ge  2\sum_{i=0}^{k-2} (2i+1)\cdot 6^{i+1} \cdot (g-(2i+2)) + 2\sum_{i=1}^{k-1}  7i\cdot 6^{i} \cdot (g-(2i+1)) \\
        &= 2\frac{6^k(365k-606)+840k+606}{125}
    \end{align*}

    Let $N :=N(7,g)$. Note that $G$ has diameter bounded by $g$ (if not, take two edges at maximum distance and consider the two neighbourhoods of them, up to distance $k-1$). Hence $\diam (G') \le 2g=4k.$ 
    For any $v' \in V(G')$, we can approximate that
    \begin{align*}
        \sigma(v')= \sum_{w \in G'} d(w,v') &\le 2g\cdot 7 \cdot 0.0005 N +\sum_{i=0}^{g-1} \left( \abs{ N_{u'}^i}(i+1)+ i \abs{ N_{v'}^i} \right)\\
    &= 0.007 g N +1+ \sum_{i=1}^{k-1} (6^i \cdot 8i) +  6^k\cdot (4k-1)\\
    &=0.014 k N + \frac{8}{25}\left( (5k-6)6^k+6) \right) + 6^k(4k-1)\\
    &< \frac{0.14k 6^k}{25}  + \frac{6^k(140k-73)+73}{25}\\
    &=\frac{6^k(700.7k-365)+365}{125}
    \end{align*}

    Here we have used that $\abs{ N_{u'}^{2i-1}}= \abs{ N_{u'}^{2i}}=\abs{ N_{v'}^{2i-1}}= \abs{ N_{v'}^{2i}}=6^i$ and $(2i-1)+2\cdot 2i +(2i+1)=8i$ when $1 \le i \le k-1.$

    Finally, we notice that 
    $\diff(G',v')-\sigma(v')>10^3 k^2 $
    since $(29.3k-847)6^k-1680k-847>10^6 k^2$ for $k \ge 29$ and that $W(G'\setminus v')-W(G')=\diff(G',v')-\sigma(v').$

    The proof will now be essentially finished by the following claim.
    
\begin{claim}
    Let $G'$ be a vertex-transitive graph and satisfy $W(G'\setminus v')-W(G')=x>16k^2$ for some $k \in \mathbb N.$
    Then we can append a tree to $u'v'$ of order bounded by $\sqrt{2x}$ to form a graph $H$ satisfying $W(H \setminus w)-W(H )=0$ for every $w \in V(G')$ for which $d_{G'}(u'v', w)=2k-1$.    
\end{claim}
\begin{claimproof}
    Choose the largest $\ell$ such that $\binom{\ell}{2}-\binom{2k}{2} \le x.$
    Note that $\ell \ge 4k.$
    Let $y=x+\binom{2k}{2}-\binom{\ell}{2}.$
    If $y=0$ or $2k+1 \le y,$
    we append a path $P_{\ell}$ to $G'$ by connecting one end vertex of the path to both $u'$ and $v'$ and connect an additional pendent vertex to the path, such that its distance to $u'$ and $v'$ equals $y-(2k-1).$
    If $0<y<2k+1,$ append a path $P_{\ell-1}$ to $G'$ as before, and add two vertices such that the sum of their distances to $u'$ and $v'$ equals $\ell-1+y-2(2k-1).$
    Let the added tree be denoted with $T$.
    Now for every $w \in V(G')$ with $d_{G'}(u'v', w)=2k-1$, $$W(H \setminus w)-W(H )= W(G'\setminus w)-W(G')- \sum_{z \in V(T)} d_{H}(w,z)=0. \qedhere$$ 
\end{claimproof}
Since $x$ is clearly bounded by $\sigma_{G'}(w)=\sigma(w')<2g \cdot 8N,$
the number of vertices in $T$ is small.

There are $2 \cdot 6^k$ vertices $w$ in $G'$ for which $d_{G'}(u'v',w)=2k-1$.
Now $\frac{2\cdot 6^k}{ \abs{ V(H)}} \ge \frac{2\cdot 6^k}{7.0035N(7,g) + \abs{V(T) }} \sim \frac{5}{7.0035}.$ 
\end{proof}

So far, $8$ explicit examples of $\frac 23$- \v Soltes' graphs have been found after performing multiple searches.
They are listed and verified in~\cite[doc. \text{2\_3\_Soltesgraphs}]{C24}.

They contain the example on $12$ vertices in~\cref{fig-good-vertices} (left) mentioned in~\cite{AOVVVY23, DV24},
an example~\cref{fig:Soltes2/3_SM} on $69$ vertices first discovered by Kurt Klement Gottwald and Snje{\v{z}}ana Majstorovi{\'c} Ergoti{\'c} (some years ago) and an example~\cref{fig:Soltes2/3_JJ} on $60$ vertices found by Jorik Jooken.

The ones of order $69$ and $384$ are derived from subdiving some edges of vertex-transitive graphs. 
The ones of order $60,90,180,300$ by adding new vertices which are connected to the end vertices of a perfect matching (an edge orbit) of the cubic vertex-transitive graphs which can be found on \url{https://houseofgraphs.org/graphs} by their graph Id's $36462, 36702, 38064$ and $41433$.

\section{Counterarguments to the present conjectures on \v Solt\'es' graphs}\label{sec:counterarg}

When~\cite[conj.~4.2]{BKS23} suggested that $C_{11}$ may be the only \v Solt\'es' graph, they did so based on the limited progress on fractional \v Solt\'es' graphs.
In~\cite{AOVVVY23} and~\cite{DV24}, there has been proven that there are infinitely many graphs for which roughly half of the vertices are \v Solt\'es' vertices. This seems the best one can do when aiming for an infinite family of graphs for which the \v Solt\'es' vertices span a $2$-regular graph and the graph has an analogous simple structure.
A few examples of $\frac 23$ and plausibly higher (conditional proof with the conditions being related to the hard problem on determining the order of cages, see e.g.~\cite{EJMS19}) were given in~\cref{sec:conditionalconstruction}.
See~\cref{fig-good-vertices} for the only two graphs of order bounded by $12$ different from $C_{11}$, with at least half of the vertices being \v Solt\'es' vertices.
Depending on the beliefs of the behaviour for large order, one may even expect that there also graphs with a $\frac 89$-fraction of the vertices being \v Solt\'es' vertices (considering the arc-graph of an arc-transitive $8$-regular graph and connect the vertices of each $K_8$ with a new vertex).

 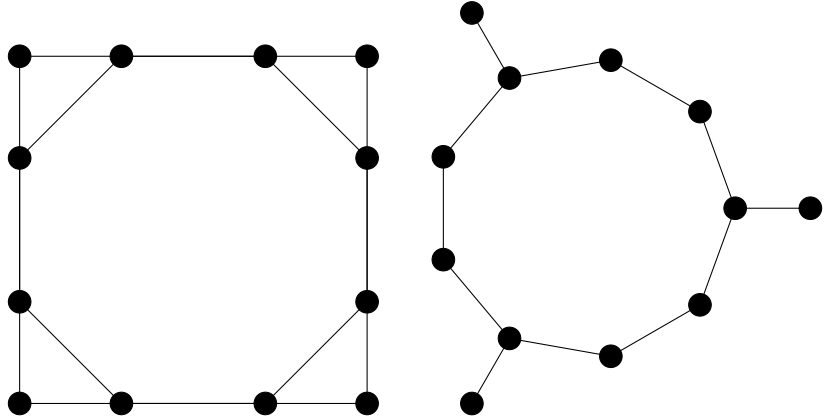
\begin{figure}[h!]
 \centering
 \begin{tikzpicture}
  \foreach \x in {0,45,...,315}{\draw[fill] (22.5+\x:2.5) circle (0.15);
\draw(\x+67.5:2.5) -- (\x+22.5:2.5);
}
 \foreach \x in {45,135,225,315}{
 \draw[fill] (\x:3.26640741219) circle (0.15);
 \draw(\x+90:3.26640741219) -- (\x:3.26640741219);
 }

\end{tikzpicture}\qquad
\begin{tikzpicture}
 \foreach \x in {0,40,...,320}{\draw[fill] (\x:2) circle (0.15);
\draw(\x+40:2) -- (\x:2);
}
 \foreach \x in {0,120,240}{
\draw[fill] (\x:3) circle (0.15);
\draw(\x:3) -- (\x:2);
 }

\end{tikzpicture}

\caption{The unique graphs on 12 vertices with 8 and 6 \v Solt\'es' vertices, respectively.}
\label{fig-good-vertices}
\end{figure}

The other argument in~\cite{BKS23} was the verification of some small vertex-transitive graphs, none of them being a \v Solt\'es' graph. Noting that $W(C_n)=W(P_{n-1})$ implies $n=11$, one can focus on the graphs with minimum degree at least $3.$
The following (near-)corollary of ~\cref{prop:diam_ge4} and~\cref{prop:diam_ge3} indicates that the transmission of a vertex is larger than the sum of few additional detours among the other pairs for the dense graphs.

\begin{cor}\label{cor:10^8fails}
    There is no vertex-transitive graph $G$ of order $n\le 47$ and $\delta(G) \ge 3$ for which $W(G)\le W(G \setminus v)$ for $v \in V(G).$
\end{cor}

\begin{proof}
    By~\cref{prop:diam_ge4} and~\cref{prop:diam_ge3}, we can exclude the examples with diameter bounded by $3.$
    The remaining examples (around $17000$) from the census~\cite{HR20} can be verified as well by direct verification.
    Since the vertex-connectivity of a $d$-regular vertex-transitive graph is at least $ \frac{2}{3}(d+1)$ (see proof of~\cite[Thm.~3]{Watkins70}) and thus $n \ge \sum_{1 \le i \le 4} \abs{N_i}\ge 2(d+1)+\frac{2}{3}(d+1),$
    also the diameter did not have to be checked for all graphs of the census to deduce the list of the graphs with diameter at least $4.$ This is done in~\cite[doc. \text{VT47}]{C24}.
\end{proof}

By~\cite[doc. \text{arcgraph10000}]{C24}, we know meanwhile that there are many vertex-transitive graphs for which $\delta(G)>2$ and $W(G) \le W(G \setminus v)$ and the argument of~\cref{cor:10^8fails} thus does not extend. 

Next, we prove that there are counterexamples to some of the most natural conjectures for the variant of~\cref{prob:extensionofSoltes}, by giving two families of non-regular graphs (which are thus not vertex-transitive nor a Cayley graph) for certain $z.$

\begin{prop}
    There are non-regular graphs $G$ and values $z$ satisfying $W(G)-W(G\setminus v)=z$ for every $v \in V(G).$
    Hereby the number of orbits or $\Delta-\delta$ can be unbounded.
\end{prop}
    
\begin{proof}
    Our first example consists of a circular ladder $CL_{2k+4}$ (two cycles $C_{2k+4}$ connected with a perfect matching), for which we add replace two opposite $C_4$s by $K_4$s (i.e., add $4$ additional edges). An example for $k=1$ is depicted in~\cref{fig:nonreg_exam} (left). This graph has $\ceilfrac k2 +1$ many vertex orbits.
    The initial circular ladder is a planar graph, which can be drawn with an inner- and outer cycle $C_{2k+4}$.
    Removing the edges of the two $K_4$s, we would obtain two components. We call the vertices of these components the left and the right vertices.
    The transmission for all vertices is the same.
    For any left vertex of the outer cycle, the sum towards the other vertices on the outer cycle, is fixed (no shortest paths use vertices at the inner cycle).
    The ones to the inner cycle are within the same distance as their copies on the outer cycle, except that for the left ones the distance is one larger.
    Thus $\sigma(v)=\floorfrac{(2k+4)^2}{4}+(k+2).$
    Let the outer left vertices be $v_0$ (a $4$-vertex) up to $v_{k+1}.$
    Let the outer right vertices be $w_0 \ldots w_{k+1}$ in a symmetric way.
    Let $v'_i$ and $w'_i$ denote the inner vertices.
    Now removing the vertex $v_i$, implies that the distance between $v_j$ and $v_m$ where $0 \le j <i<m \le k+1$ increases by two.
    The distance between $v_j$ and $w_m$ (or $w'_m$) increases by one, if $j+m>k+1$ and $j<i$.
    Analogously for $v_m$ and $w_j$ when $m>i.$
This implies that the sum of increases in the distance equals 
    $$2i(k+1-i)+2 \sum_{j=1}^{i-1} j + 2\sum_{j=1}^{k-i}=k^2+k.$$
Since both the transmission and difference in distances is independent of $i$, and by symmetry between left and right, as well as inner and outer, we conclude.

    Now we consider a second family, where the irregularity can be arbirarily large, i.e. $\Delta(G)-\delta(G)=k$ for every $k>0.$
    Let $n=2k+2$.
Let $A_0$ and $A_3$ be sets of $k$ vertices, $B_1$ and $B_2$ be both sets of order $3n$ spanning $n$ triangles.
Let $G[A_0, B_1]$ and $G[A_3, B_2]$ be complete bipartite.
Here $G[U,V]$ denotes the subgraph of $G$ spanned by $U \cup V$ with only edges $uv$, where $u \in U, v \in V$, present.
For every pair of a triangle $T_1$ in $B_1$ and a triangle $T_2$ in $B_2$, add a $C_6$ in between (i.e., let $G[T_1, T_2]\cong C_6$).
It is easy to see that every $2$ vertices have at least two disjoint shortest paths in between them.
To conclude, we only need to prove that the transmission of every vertex $v$, $\sigma(v)=\sum_{u \in V} d(v,u)$, is equal.
For a vertex $v$ in $A_0, A_3$, since it has eccentricity $3$, we easily compute that $\sigma(v)=\sum_{i=1}^3 i \cdot \abs{N_i(v)} = 3n + 2 \cdot (3n+k-1) + 3 \cdot k= 9n+5k-2.$ Here $N_i(v)=\{u \in V \mid d(u,v)=i\}.$
Every vertex in $B_1 \cup B_2$ has eccentricity $2$ and degree $k+2+2n=3n-k.$
Its transmission thus equals $2(6n+2k-1)-(3n-k)=9n+5k-2=23k+16.$
An example for $k=1$ is depicted in~\cref{fig:nonreg_exam} (right).
\end{proof}

\begin{figure}[h]
    \centering
    \begin{tikzpicture}[scale=0.9]

     \foreach \z in {1,2,5,4}{
     \foreach \y in {1,2}{
     \draw[fill] (\y,\z) circle (0.1);
     }
 }

 \foreach \z in {1,4}{
     \foreach \y in {1,2}{
     \draw(\y,\z) --(\y,\z+1);
     }
 }

 \foreach \y in {0,-1,3,4}{
     \draw[fill] (\y,3) circle (0.1);
     }
      \foreach \y in {-1,3}{
     \draw (\y,3) --(\y+1,3);
     }

     \foreach \z in {1,2,5,4}{
     \draw (1,\z)--(2,\z);
 }

 \draw (-1,3)--(1,5)--(2,4)--(3,3)--(2,2)--(1,1)--(-1,3);

 \draw (0,3)--(1,4)--(2,5)--(4,3)--(2,1)--(1,2)--(0,3);

    \end{tikzpicture}\quad
    \begin{tikzpicture}[scale=0.6]

 \foreach \z in {1,2,...,12}{
     \foreach \y in {1,2,...,12}{
     \draw[black!20!white] (5,\z) -- (10,\y);
     }
 }

  \foreach \z in {1,2,...,12}{
     \draw (5,\z) -- (0,6.5);
     \draw (10,\z)--(15,6.5);
     }

  \foreach \z in {1,4,...,10}{
     \draw (10,\z) -- (10,\z+2);
     \draw (5,\z) -- (5,\z+2);
     }

  \foreach \z in {1,4,...,10}{
  \draw[dashed] (10,\z) arc (-30:30:2cm);
     \foreach \y in {1,4,...,10}{
     \draw[red!70!white] (5,\z) -- (10,\y);
     }
 }

   \foreach \z in {2,5,...,11}{
   
     \foreach \y in {2,5,...,11}{
     \draw[red!70!white] (5,\z) -- (10,\y);
     }
 }

   \foreach \z in {3,6,...,12}{
   \draw[dashed] (5,\z) arc (150:210:2cm);
   
     \foreach \y in {3,6,...,12}{
     \draw[red!70!white] (5,\z) -- (10,\y);
     }
 }

            \draw[fill] (0,6.5) circle (0.15);
        \draw[fill] (15,6.5) circle (0.15);
        
        \foreach \x in {1,2,...,12}{
        \draw[fill] (5,\x) circle (0.15);
        \draw[fill] (10,\x) circle (0.15);
        }
    \end{tikzpicture}
    \caption{Non-regular graphs on $12$ and $26$ vertices resp. for which $W(G)-W(G \setminus v)$ is independent of $v$}
    \label{fig:nonreg_exam}
\end{figure}
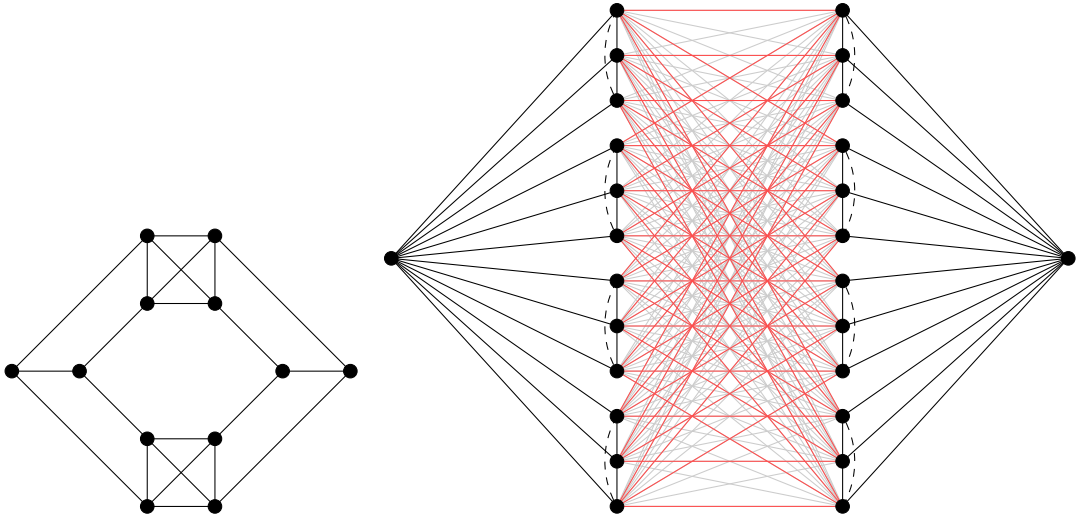

Nevertheless, these examples are constructed in such a way that the distances in $G \setminus v$ barely change, which implies that $W(G)-W(G\setminus v)$ is large.
We give an example of a non-vertex-transitive graphs for which $W(G)-W(G\setminus v)$ is fixed and small compared with its order.

\begin{exam}
    The arc-graph of \url{https://houseofgraphs.org/graphs/19281} has order $320$ is non-vertex-transitive graphs and satisfies $W(G)-W(G\setminus v)=40$ for every $v \in V(G).$
    This is verified in~\cite[doc. \text{Nonvertex-transitive}]{C24}.
\end{exam}

Checking graphs with two orbits, there are many for which $W(G)-W(G \setminus v)$ is independent of $v$.
Almost all connected large graphs with few orbits are at least $2$-connected.
If there is no restriction on the sign of $W(G)-W(G \setminus v)$, as this difference is bounded by $-n^3$ and $n^3$, one may expect that there are both vertex-transitive and non-vertex-transitive (and thus not Cayley graphs) \v Solt\'es' graphs. The latter is the case if there are $d$-regular (for $3 \le d \le 7$) arc-transitive graphs with large girth $g$ near the Moore bound $N(d,g)$.



\section{Conclusion}\label{sec:conc}

The problem of \v Solt\'es, due to its charming simplicity, has intrigued many people and resulted in over a $100$ citing papers in the past. 
In this paper, we gave indications why some of the current conjectures were not well-founded and may be false.
Furthermore, the research lead to the following key question.

\begin{q}\label{q:keyques}
    Are there infinitely many negative-\v Solt\'es' graphs?
\end{q}


We conjecture that if this question has an affirmative answer
, for every $z \in \mathbb Z,$ there are infinitely many graphs $G$ for which $W(G)-W(G \setminus v)=z$ for every $v \in V(G)$ and among them, there are also non-vertex-transitive ones.
As an analogy, we refer to the conjecture of Heath-Brown on solutions of the Diophantine equation $x^3+y^3+z^3=k$, which says that there are infinitely many solutions whenever $k \not \equiv \pm 4 \pmod 9$ (i.e., when there is no simple reason of non-existence). Also for this somewhat computationally simpler question, it has taken a while to find large solutions $(x,y,z) \in \mathbb Z^3$ for some values $k$, see e.g.~\cite{BS20}.

If there is a value $n_0$ such that every graph $G$ of order $n \ge n_0$ satisfies $W(G)>W(G \setminus v)$ for $v \in V(G)$, (no negative-\v Solt\'es' graph exists with order at least $n_0$), we know that $n_0$ is much larger than $10^4$ by~\cite[doc. \text{arcgraph10000}]{C24}.
Under the existence of $n_0$, addressing the problem of \v Solt\'es will need ideas to determine or exclude the \v Solt\'es' graphs with order bounded by $n_0.$
When~\cref{q:keyques} has a negative answer, the number of solutions for~\cref{prob:extensionofSoltes} may be approximately increasing over positive $z$.

Related to~\cref{q:keyques}, one can also ask the extremal question on a good lower bound for $W(G)-W(G \setminus v)$ for vertex-transitive graphs of order $n$ and degree $d$.
Checking the census on cubic vertex-transitive graphs up to $1280$ vertices~\cite{POSV13}, we note that there are only a few dozens of negative-\v Solt\'es' graph among them and all of these are arc-graphs, giving a direction for the extremal question.
Hereby we formulate the following related question for general graphs.
\begin{q}
    A graph $G$ for which $W(G) \le W(G \setminus v)$ for all $v\in V(G)$ has minimum degree upper bounded by $7?$
\end{q}

\section*{Acknowledgments}
The author thanks Marston Conder, Jorik Jooken, Snje{\v{z}}ana Majstorovi{\'c} Ergoti{\'c} and Dragan Stevanovi\'c for discussions and informing on references and state-of-the-art results.

\bibliographystyle{abbrv}
\bibliography{ref}

\appendix

\section{Presentation of examples of $\frac 23$-\v Solt\'es' hypergraphs}\label{sec:app}

The smallest new example is a graph formed by a $C_{40}$ on $[40]$, with additional edges between $i$ and $i+11$ if $i \equiv 1 \pmod 4$
and between $i$ and $i+9$ if $i \equiv 2 \pmod 4$, together with an additional vertex connected to $i$ and $i+1$ for $2 \mid i.$
Here indices are considered modulo $40$.

\begin{figure}[h]
    \centering
    \begin{tikzpicture}[scale=0.7]

\foreach \x in {0,1,2,...,39}{
\draw (\x*9:8) -- (9*\x+9:8) ;
\draw[fill] (\x*9:8) circle (0.15);
}

\foreach \x in {0,1,2,...,19}{
\draw (\x*18:8) -- (\x*18+4.5:9)-- (18*\x+9:8) ;
\draw[fill=red] (\x*18+4.5:9) circle (0.15);
}

\foreach \x in {0,1,2,...,9}{
\draw (36*\x+9:8) -- (36*\x+108:8) ;
\draw (36*\x+18:8) -- (36*\x+99:8) ;
}

\end{tikzpicture}
    \caption{A graph of order $60$ with $40$ \v Solt\'es' vertices}
    \label{fig:Soltes2/3_JJ}
\end{figure}
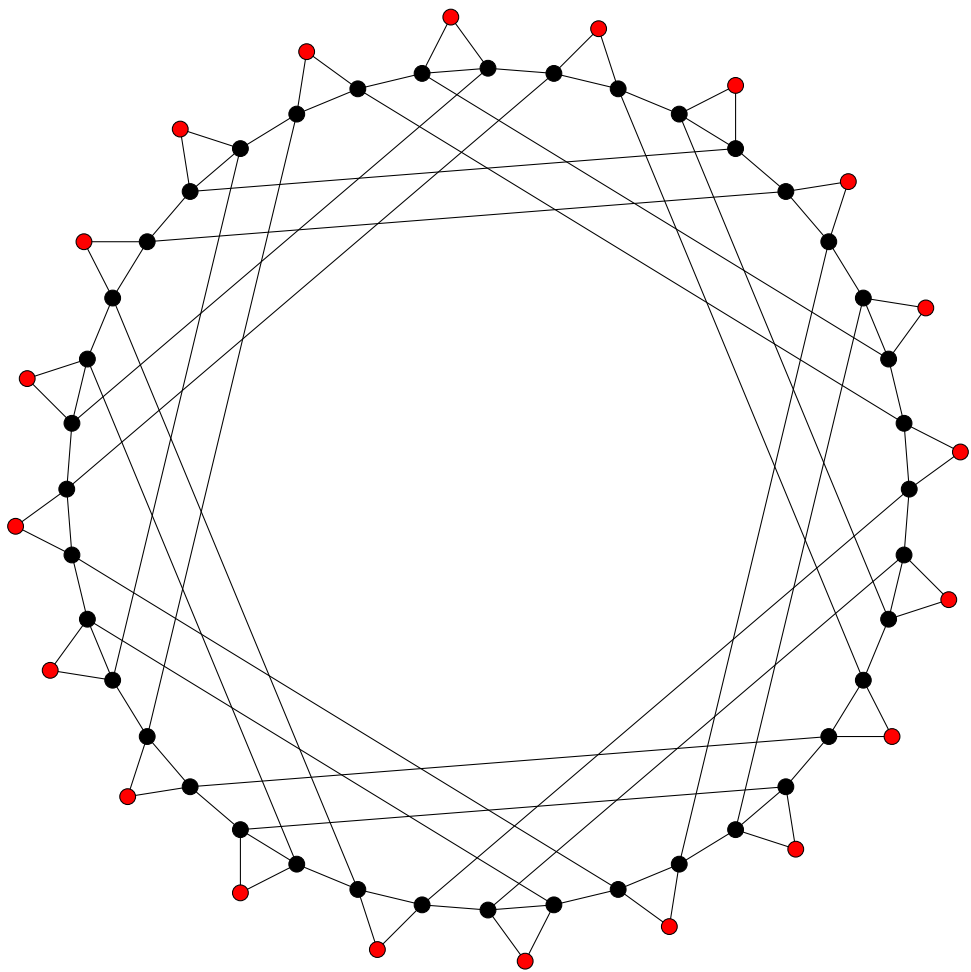

\begin{figure}[h]
    \centering
    \begin{tikzpicture}[scale=0.7]

\foreach \x in {0,1,2,...,21}{
\draw (\x*360/23:8) -- (360/23*\x+360/23:8) ;
\draw (\x*360/23:6) -- (360/23*\x+360/23:6) ;
}
\foreach \x in {0,1,2,...,21,22}{
\draw[fill] (\x*360/23:8) circle (0.15);
\draw[fill] (\x*360/23:6) circle (0.15);
\draw (\x*360/23:6) -- (\x*360/23:8);
\draw[fill=red] (\x*360/23:7) circle (0.15);
}

\draw (22*360/23:8) -- (360/23*23:6) ;

\draw (22*360/23:6) -- (360/23*23:8) ;

\end{tikzpicture}
 \quad
    \begin{tikzpicture}[scale=0.7]

\foreach \x in {0,1,2,...,21,22}{
\draw (\x*360/23:6) -- (360/23*\x+360/23:8) ;
\draw (\x*360/23:8) -- (360/23*\x+360/23:6) ;
\draw (\x*360/23:6) -- (\x*360/23:8);
\draw[fill] (\x*360/23:8) circle (0.15);
\draw[fill] (\x*360/23:6) circle (0.15);
\draw[fill=red] (\x*360/23:7) circle (0.15);
}

\end{tikzpicture}
    \caption{A graph of order $69$ with $46$ \v Solt\'es' vertices ($2$ representations)}
    \label{fig:Soltes2/3_SM}
\end{figure}
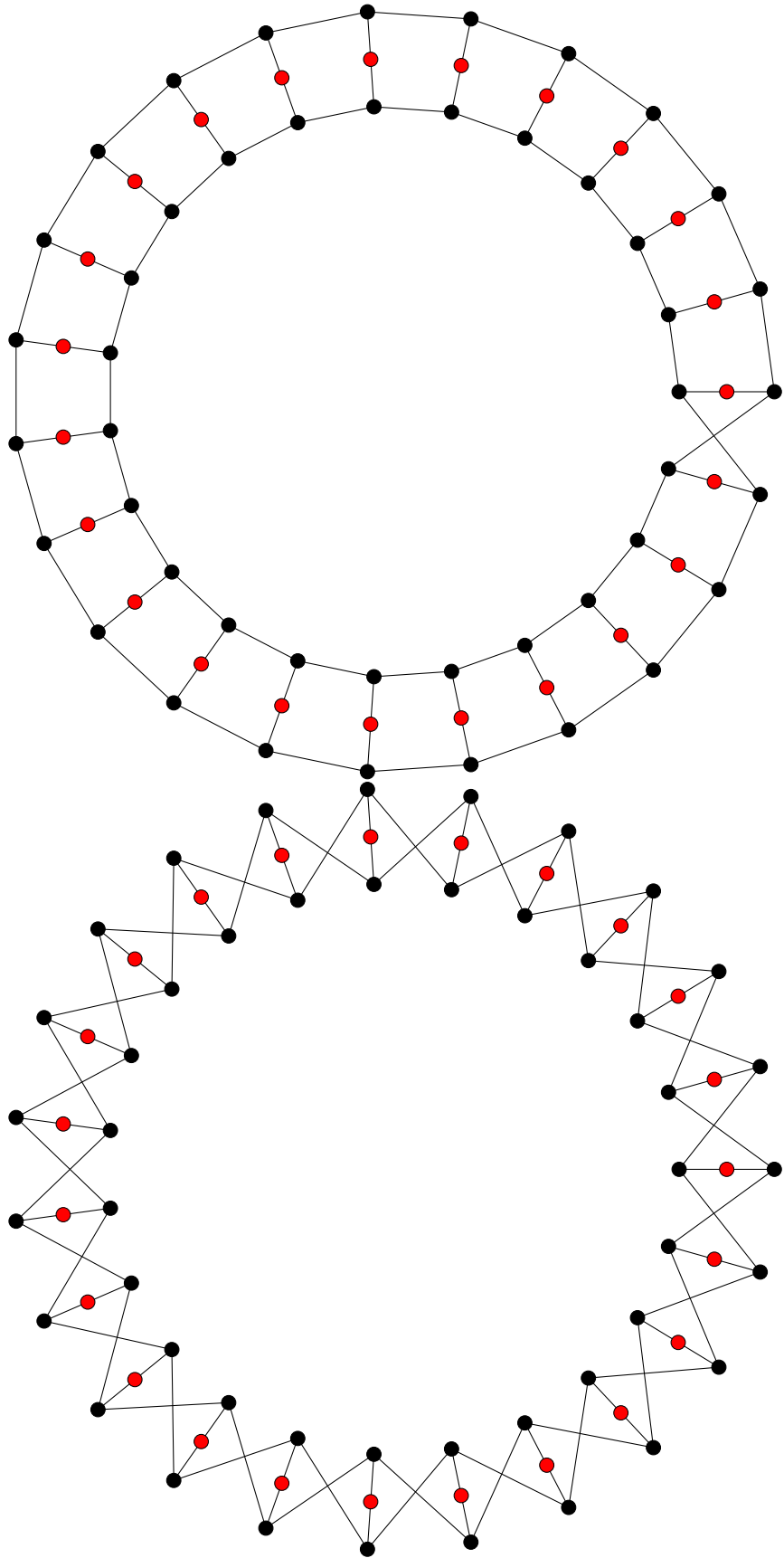

\begin{figure}[h]
    \centering
    \begin{tikzpicture}[scale=0.8]

\foreach \x in {0,1,2,...,14}{
\draw (24*\x+6:8) -- (24*\x+12:7) -- (24*\x+18:7)-- (24*\x+24:8)--(24*\x+30:8);
\draw[fill] (24*\x+6:8) circle (0.15);
\draw[fill] (24*\x+12:7) circle (0.15);
\draw[fill] (24*\x+18:7) circle (0.15);
\draw[fill] (24*\x+24:8) circle (0.15);
}

\foreach \x in {0,1,2,...,14}{
\draw[dotted] (24*\x+6:8) -- (24*\x+15:9)-- (24*\x+24:8) ;
\draw[dotted] (24*\x+12:7) -- (24*\x+15:7.5)-- (24*\x+18:7) ;
\draw[fill=red] (\x*24+15:7.5) circle (0.15);
\draw[fill=red] (\x*24+15:9) circle (0.15);
}

\foreach \x in {0,1,2,...,15}{
\draw[dashed] (24*\x+6:8) -- (24*\x+24:8) ;
\draw[dashed] (24*\x+12:7) -- (24*\x+114:7) ;
}

\end{tikzpicture}
    \caption{A graph of order $90$ with $60$ \v Solt\'es' vertices}
    \label{fig:Soltes2/3_SC}
\end{figure}
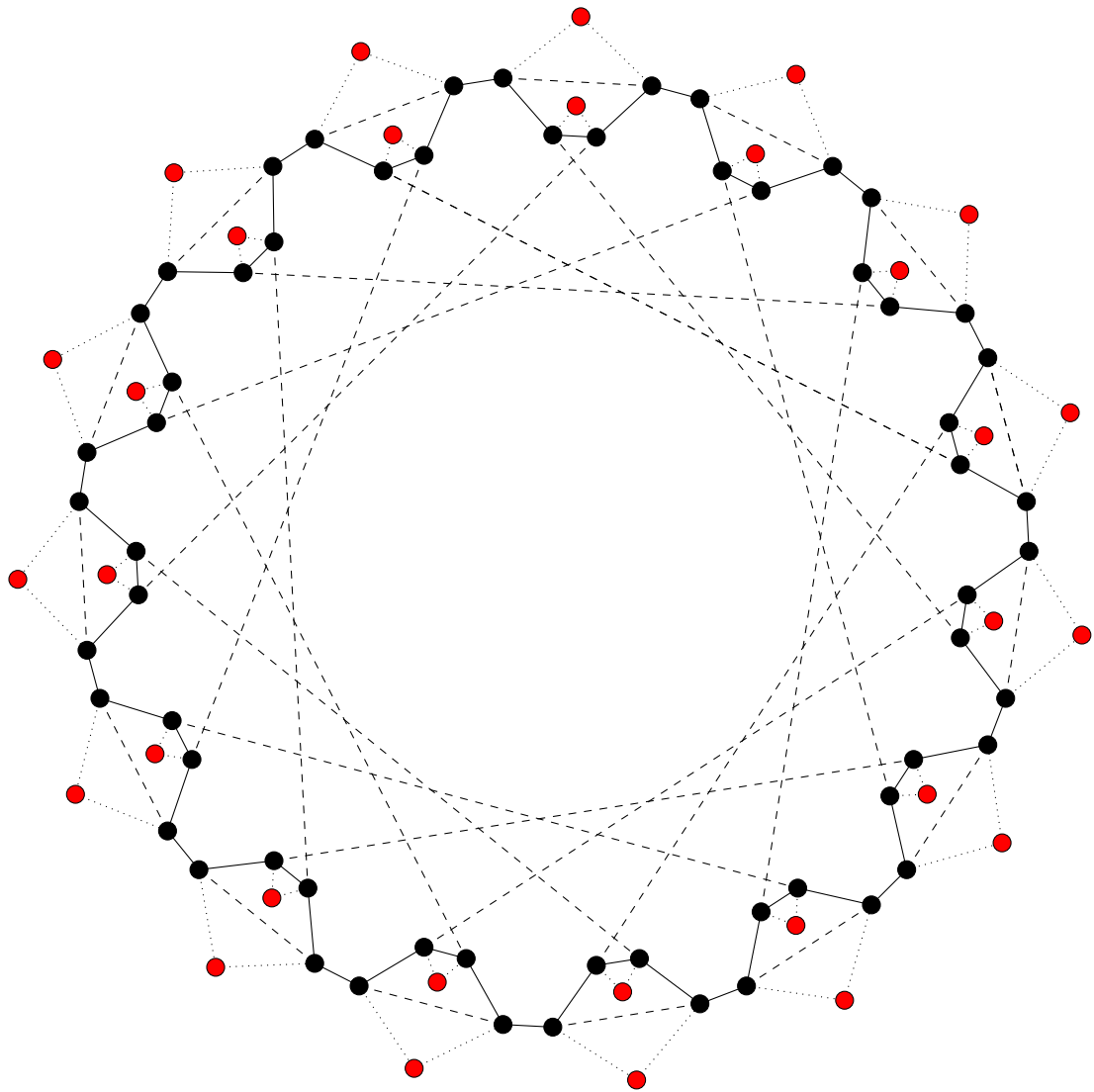

\end{document}